\documentclass[11pt]{article}

\usepackage[margin=1in]{geometry}
\usepackage[T1]{fontenc}
\usepackage[utf8]{inputenc}
\usepackage{lmodern}
\usepackage{microtype}
\usepackage{amsmath,amssymb,amsthm,mathtools}
\usepackage{booktabs,array}
\usepackage{graphicx}
\usepackage{enumitem}
\usepackage{algorithm}
\usepackage[noend]{algpseudocode}
\usepackage[hidelinks]{hyperref}
\usepackage{doi}
\usepackage{authblk}
\usepackage[backend=biber,citestyle=numeric-comp,bibstyle=ieee,sorting=none,minbibnames=5,maxbibnames=8,giveninits=true]{biblatex}
\newtheorem{theorem}{Theorem}[section]
\newtheorem{proposition}[theorem]{Proposition}
\newtheorem{lemma}[theorem]{Lemma}
\newtheorem{corollary}[theorem]{Corollary}

\theoremstyle{remark}
\newtheorem{remark}[theorem]{Remark}

\DeclareMathOperator{\perm}{perm}

\newcommand{\cS}{\mathcal S}

\newcommand{\defeq}{:=}

\title{Constructive recurrences for determinants and permanents of banded Toeplitz matrices}
\author[1]{Max A. Alekseyev}
\author[2]{Dmitry I. Khomovsky}
\affil[1]{\small The George Washington University, Washington, DC, USA. Email: \href{mailto:maxal@gwu.edu}{maxal@gwu.edu}}
\affil[2]{\small Email: \href{mailto:khomovskij@physics.msu.ru}{khomovskij@physics.msu.ru}}
\date{}

\begin{document}
\maketitle

\begin{abstract}
For fixed nonnegative integers $m_1,m_2$, let $A_n=(a_{j-i})_{i,j=1}^n$
be the leading $n\times n$ section of a Toeplitz matrix with lower and upper
semibandwidths $m_1$ and $m_2$.  We give two constructive Laplace-expansion
methods for scalar recurrences of $\det A_n$ and $\perm A_n$.  The
increasing-rows method eliminates a fixed family of boundary cofactors and
gives recurrence order at most $d=\binom{m_1+m_2}{m_1}$ for both
sequences.  The row-column method closes normalized boundary minors
recursively and packages them in a sparse transfer matrix.  Its reachable
states are classified exactly: level $j$ is indexed by a pair of $j$-subsets
of $[m_1]$ and $[m_2]$.  Hence the transfer dimension is $d$, and we obtain
an explicit formula for the number of nonzero transitions.

For determinants, the complementary cofactors of the increasing-rows
construction are coordinates of the classical compound companion
representation.  The independently constructed row-column transfer has the
Widom characteristic polynomial and is generically similar to the compound
transfer.  Thus the order $d$ recurrence is generically minimal for the
unrestricted fixed-band determinant family.  For permanents the same state
graph gives the binomial upper bound, without a general minimality claim.
The pentadiagonal case recovers Sweet's order-six determinant recurrence and
its permanent analogue, while the one-superdiagonal family gives closed
scalar recurrences and rational generating functions.  Position-dependent
band weights preserve the finite state graph but replace the constant
transfer by a cocycle.  For cyclic determinants, Fourier diagonalization
produces all subset products of the symbol roots and a generically minimal
annihilator of degree $2^{m_1+m_2}$, corresponding to the passage from one
exterior degree to the full exterior algebra.
\end{abstract}



\section{Introduction}

Let
\[
 A_n=(a_{j-i})_{i,j=1}^n
\]
be the $n\times n$ Toeplitz matrix generated by the Laurent polynomial
\[
 a(z)=\sum_{s=-m_1}^{m_2}a_s z^s,
 \qquad a_{-m_1}a_{m_2}\ne0,
\]
where $m_1,m_2$ are fixed nonnegative integers.  We call $m_1$ and $m_2$
the lower and upper \emph{semibandwidths}; the total bandwidth is
$m_1+m_2+1$.  We write
\[
 D_n(m_1,m_2)\defeq\det A_n,
 \qquad
 P_n(m_1,m_2)\defeq\perm A_n,
\]
and suppress $(m_1,m_2)$ when no confusion can arise.

Finite recurrence structure already appears in the pentadiagonal case
$(m_1,m_2)=(2,2)$.  Sweet \cite{Sweet} proved the order-six relation
\begin{align}
D_{n+6}={}&a_0D_{n+5}-(a_{-1}a_1-a_{-2}a_2)D_{n+4}
 +(a_{-2}a_1^2+a_{-1}^2a_2-2a_{-2}a_0a_2)D_{n+3} \notag\\
&-a_{-2}a_2(a_{-1}a_1-a_{-2}a_2)D_{n+2}
 +a_0(a_{-2}a_2)^2D_{n+1}-(a_{-2}a_2)^3D_n, \label{eq:sweet}
\end{align}
Closed Chebyshev forms for general five-diagonal Toeplitz determinants
were obtained by Marr and Vineyard \cite{MarrVineyard1988}, while Hadj and
Elouafi \cite{HadjElouafi2008} treated characteristic polynomials,
eigenvectors, and determinants for pentadiagonal matrices.  Related
recurrence and determinant formulas have appeared in several later forms;
see, for example, \cite{Cinkir,Jia,An,DuFonsecaPereira2022}.  The existence
of finite recurrences in arbitrary fixed bandwidth is classical.  Widom's
formula \cite{Widom1958} expresses a banded Toeplitz determinant as a finite
sum of exponential terms whose bases are fixed-cardinality products of the
roots of the symbol.  The natural number of such characteristic factors is
\[
 d(m_1,m_2)\defeq\binom{m_1+m_2}{m_1}.
\]
The finite-section determinant theory sits inside the broader spectral theory
of banded Toeplitz matrices; a convenient modern reference is
B\"ottcher and Grudsky \cite{BottcherGrudsky2005}, whose treatment includes a
chapter devoted to finite Toeplitz determinants.  Complementary matrix-
theoretic treatments include Bini and Capovani \cite{BiniCapovani1983},
Trench \cite{Trench1985}, and Tismenetsky \cite{Tismenetsky1987}.
Alexandersson \cite{Alexandersson2012} later placed Widom's formula and
recurrences for Toeplitz minors in the language of Schur polynomials.

There is also a substantial recurrence literature for permanents.  Minc's
early work on $(0,1)$-circulants \cite{Minc1964} was followed by explicit
recurrence constructions \cite{Minc1985} and by a permanental-compound
interpretation \cite{Minc1987}.  Shevelev's recurrence program treats both
circulant and Toeplitz permanents and determinants
\cite{SheFund1990,She,SheSurvey1992,She1}; on the determinant side his
algomatrix is an associated (compound) matrix of companion type.  Golin,
Leung, and Wang \cite{GolinLeungWang2006} give a combinatorial transfer-
matrix derivation of fixed-jump circulant recurrences by encoding partial
cycle covers with finite two-sided boundary data.  Sparse Toeplitz and
circulant permanents have also been studied from several complementary
viewpoints \cite{Cod,Schwartz2009,Koch,Kamenetskii2007,CodenottiResta2002}.
Zakraj\v{s}ek and Petkov\v{s}ek \cite{Zakr} prove recurrence results for
principal minors of arbitrary banded matrices with varying entries.

Our contribution is a constructive description of the recurrence mechanism
and an exact comparison of several natural finite-dimensional realizations.  We develop two Laplace-expansion
constructions.  The first, which we call the \emph{increasing-rows method},
expands a short family of consecutive sections against a common collection
of boundary cofactors and eliminates those cofactors.  It gives the binomial
upper bound $d(m_1,m_2)$ simultaneously for determinants and permanents.
The second, the \emph{row-column method}, recursively expands exposed
boundary rows and columns, normalizes the resulting minors by translation,
and records the closed family as a sparse transfer matrix.  The same
signature graph supports determinants and permanents; only the Laplace signs
change.

The row-column construction exposes more structure than a bare finiteness
proof.  We classify every reachable normalized signature: level $j$ is
indexed by a pair of $j$-subsets of $[m_1]$ and $[m_2]$.  Vandermonde's
identity then gives the exact total state count
\[
 \sum_j\binom{m_1}{j}\binom{m_2}{j}
 =\binom{m_1+m_2}{m_1},
\]
and the same classification yields a closed sparsity formula for the
transition matrix.  In the pentadiagonal case this produces a six-state
system whose characteristic polynomials give Sweet's determinant recurrence
and the corresponding permanent recurrence.  The balanced $(3,3)$ transfer,
with $20$ states and $50$ nonzero entries, is displayed in the appendix as a
larger concrete example.

For determinants we can compare the two constructions with the classical
root-product picture exactly.  The complementary cofactors generated by the
increasing-rows method are coordinates of the fixed exterior power of a
companion matrix.  The independently constructed row-column transfer has the
same characteristic polynomial as the Widom annihilator and is generically
similar to the compound transfer.  In this sense the sparse boundary-state
graph, Shevelev's associated-matrix construction, and Widom's products of
symbol roots are three coordinate descriptions of the same finite-dimensional
object.  This comparison also proves that the binomial recurrence order is
generically minimal for the unrestricted determinant family.  We make no
parallel generic-minimality assertion for permanents: the constructive
binomial upper bound is general, while minimal permanent recurrences may
collapse under additional structure.

A recent preprint of Hendel \cite{Hendel2026} provides another systematic
Laplace-expansion procedure for banded square Toeplitz determinant families.
Its stages are organized by successive deletion of leading rows and boundary
column sets, whereas our increasing-rows method eliminates a fixed family of
complementary cofactors and the row-column method uses exchanged boundary
subsets.  In the balanced $(m,m)$ case the total number of pre-closing Hendel
patterns is nevertheless the same central binomial coefficient.  This
agreement is useful context for the binomial scale, but the state
stratifications and resulting transfer descriptions are different.  We compare
these realizations side by side in Section~\ref{sec:laplace-comparison}.

Two later sections separate two genuinely different ways of leaving the
autonomous open-Toeplitz setting.  Changing the coefficients while keeping
open boundaries preserves the finite local state architecture: for leading
principal sections of a fixed position-dependent banded matrix, the signature
graph survives unchanged but the constant transfer becomes a cocycle.  Fixed
corner defects belong to the same open-boundary axis and suggest bounded
two-sided boundary data.  Cyclic closure changes something different: the
bulk symbol stays constant while the two boundaries are identified.  Fourier
diagonalization then replaces Widom's single fixed-cardinality layer by all
subsets of the symbol roots, giving a generically minimal determinant
recurrence of order $2^{m_1+m_2}$.  Classical cycle-cover transfers provide
the corresponding finite-state framework for cyclic permanents.

The one-superdiagonal Toeplitz--Hessenberg family provides a benchmark where
the scalar recurrence and generating function can be written directly.  We
return to one sparse $(2,1)$ specialization in
Section~\ref{sec:sparse-spectrum}: after cubic rescaling, its characteristic
polynomials are the threefold-symmetric $2$-orthogonal Tchebychev family of
Douak and Maroni \cite{DouakMaroni1997I}.  For positive double-band weights,
McMillen's general theorem \cite{McMillen2009} already gives the resulting
threefold root-of-unity geometry.  The zero-location results
\cite{BenRomdhane2008,LoureiroVanAssche2020} refine that geometry here by
identifying the positive radial representatives and their strict
interlacing; the known limiting zero distribution \cite{ShapiroStampach2019}
also identifies the sharp outer radius.

The paper is organized as follows.  Section~\ref{sec:prelim} records the
classical root-product and Toeplitz--Hessenberg recurrence benchmarks.
Section~\ref{sec:increasing} develops the increasing-rows construction and
identifies its complementary cofactors with the compound companion state.
Section~\ref{sec:rowcolumn} gives the row-column transfer, proves the exact
signature classification and sparsity count, compares the determinant
transfer with the compound/Widom realization, and places these constructions
alongside Hendel's staged Laplace closure.  Section~\ref{sec:sparse-spectrum}
extracts spectral consequences from a sparse Toeplitz--Hessenberg
specialization.  Section~\ref{sec:variable} treats position-dependent open
bands and finite corner defects.  Section~\ref{sec:cyclic} treats cyclic
closure separately, emphasizing the passage from one exterior degree to the
full exterior algebra.  Section~\ref{sec:conclusion} summarizes the resulting
recurrence picture and open directions.

\section{Classical recurrence structure and benchmarks}\label{sec:prelim}

For a square matrix $B$, let $B[R,C]$ denote the submatrix with row set $R$
and column set $C$.  We use the usual signed cofactors for determinants.  For
permanents it is convenient to use \emph{signless cofactors}: deleting the
same row and column sets as in a Laplace expansion, but without the
alternating sign.

\subsection{Tridiagonal prototype}\label{sec:tridiag}

The simplest instance already shows two different spectral viewpoints that
will recur below.  Let
\[
 T_n(P,Q)=
 \begin{bmatrix}
 P&Q&&&\\
 1&P&Q&&\\
 &\ddots&\ddots&\ddots&\\
 &&1&P&Q\\
 &&&1&P
 \end{bmatrix},
 \qquad \Delta_n\defeq\det\!\bigl(T_n(P,Q)\bigr),
\]
and put $\Delta_0=1$.  Laplace expansion gives
\begin{equation}\label{eq:tridiagrec}
 \Delta_n=P\Delta_{n-1}-Q\Delta_{n-2},
 \qquad \Delta_0=1,\quad \Delta_1=P.
\end{equation}
The classical spectrum of a tridiagonal Toeplitz matrix
\cite{NoschesePasquiniReichel2013} gives the eigenvalues
\[
 P+2\sqrt Q\cos\frac{\pi j}{n+1},\qquad 1\le j\le n,
\]
and therefore
\begin{equation}\label{eq:tridiagproduct}
 \Delta_n=\prod_{j=1}^{n}
 \left(P+2\sqrt Q\cos\frac{\pi j}{n+1}\right).
\end{equation}
It is useful to record this recurrence in the standard Lucas-polynomial
notation
\[
 U_0(P,Q)\defeq0,\qquad U_1(P,Q)\defeq1,\qquad
 U_n(P,Q)\defeq P U_{n-1}(P,Q)-Q U_{n-2}(P,Q),
\]
and
\[
 V_0(P,Q)\defeq2,\qquad V_1(P,Q)\defeq P,\qquad
 V_n(P,Q)\defeq P V_{n-1}(P,Q)-Q V_{n-2}(P,Q).
\]
Then
\begin{equation}\label{eq:tridiag-Lucas}
 \Delta_n=U_{n+1}(P,Q).
\end{equation}
More generally, if $W_0=A$, $W_1=B$, and
\[
 W_k=PW_{k-1}-QW_{k-2},
\]
then
\begin{equation}\label{eq:W-Lucas}
 W_k=B\,U_k(P,Q)-AQ\,U_{k-1}(P,Q)\qquad(k\ge1).
\end{equation}
Combining \eqref{eq:W-Lucas} with the finite Toeplitz spectrum gives, for
$k\ge2$,
\begin{equation}\label{eq:Wproduct}
\begin{split}
 W_k={}&B\prod_{j=1}^{k-1}
 \left(P+2\sqrt Q\cos\frac{\pi j}{k}\right)\\
 &-AQ\prod_{j=1}^{k-2}
 \left(P+2\sqrt Q\cos\frac{\pi j}{k-1}\right).
\end{split}
\end{equation}
Thus the product formula is simply the Lucas $U$-solution written through
the spectra of two consecutive tridiagonal Toeplitz sections.  In the usual
Chebyshev notation $\mathsf U_n$,
\[
 U_{n+1}(P,Q)=Q^{n/2}\mathsf U_n\!\left(\frac{P}{2\sqrt Q}\right).
\]

The $n$ eigenvalues in \eqref{eq:tridiagproduct} vary with the matrix size.
By contrast, the two fixed characteristic roots of $t^2-Pt+Q$ govern the
sequence $(\Delta_n)$ as $n$ varies.  Widom's factors below are the
higher-band analogues of the latter.

\subsection{Balance condition on determinant monomials}

\begin{lemma}\label{lem:balance}
If $\prod_{s=-m_1}^{m_2}a_s^{n_s}$ occurs as a monomial of
$D_n(m_1,m_2)$, then
\[
 \sum_{s=-m_1}^{m_2}s\,n_s=0.
\]
\end{lemma}

\begin{proof}
Every monomial in the determinant is obtained from a permutation $\pi\in
S_n$ and has the form $\prod_{i=1}^n a_{\pi(i)-i}$.  Hence
\[
 \sum_s s n_s=\sum_{i=1}^n(\pi(i)-i)=0.
\]
\end{proof}

\subsection{Widom's characteristic factors}

Put
\[
 q(z)\defeq z^{m_1}a(z)
 =a_{-m_1}+a_{-m_1+1}z+\cdots+a_{m_2}z^{m_1+m_2}.
\]
Suppose first that the roots $z_1,\ldots,z_{m_1+m_2}$ of $q$ are pairwise
distinct.  In what follows, $J$ ranges over the $m_2$-element subsets of
$\{1,\ldots,m_1+m_2\}$.  Widom's formula \cite{Widom1958}; see also
\cite{Alexandersson2012}, writes $D_n$ in the form
\begin{equation}\label{eq:widomform}
 D_n=\sum_{|J|=m_2} C_J W_J^n,
 \qquad
 W_J=(-1)^{m_2}a_{m_2}\prod_{j\in J}z_j,
\end{equation}
with explicit nonzero rational functions $C_J$ of the roots in the generic
case.  Consequently,
\begin{equation}\label{eq:widompoly}
 \chi_{m_1,m_2}(t)
 \defeq\prod_{|J|=m_2}(t-W_J)
\end{equation}
is an annihilating polynomial for $(D_n)$ of degree $d(m_1,m_2)$.  The
coefficients of $\chi_{m_1,m_2}$ are symmetric polynomials in the roots and
hence polynomial expressions in the diagonal parameters, so the recurrence
extends by specialization when roots coalesce.

\begin{remark}\label{rem:minimal}
Throughout this paper, an \emph{order-$d$ recurrence} is a homogeneous
annihilating relation involving $d+1$ consecutive terms; equivalently, its
monic annihilating polynomial has degree $d$.  The minimal recurrence may
have smaller order.  For fixed $(m_1,m_2)$, let
\[
 \mathcal P_{m_1,m_2}
 \defeq
 \left\{
 (a_{-m_1},\ldots,a_{m_2})\in\mathbb C^{m_1+m_2+1}:
 a_{-m_1}a_{m_2}\ne0
 \right\}.
\]
By the \emph{generic minimal order} for the unrestricted determinant family
we mean the minimal recurrence order on a nonempty Zariski-open subset of
$\mathcal P_{m_1,m_2}$.  Theorem~\ref{thm:rowcolumn-compound} proves that
this order is $\binom{m_1+m_2}{m_1}$.  Special parameter choices can
identify characteristic modes and lower the minimal order.  For permanents,
Theorem~\ref{thm:increasing} gives the same binomial number as a general
upper bound, but no general generic-minimality assertion is made.
\end{remark}

\subsection{One-superdiagonal Toeplitz--Hessenberg case}\label{sec:hessenberg}

When $m_2=1$, the general binomial order becomes $m_1+1$ and the scalar
recurrence is especially transparent.  Write $m=m_1$, put $D_0=P_0=1$, and
for the recurrence notation set
\[
 D_j=P_j=0\qquad(-m\le j<0).
\]
Only these first $m$ negative indices are assigned zero values.

\begin{proposition}\label{prop:hessenberg}
For an $(m,1)$-banded Toeplitz matrix and $n\ge1$,
\begin{align}
 D_n&=\sum_{r=0}^{m}(-1)^r a_{-r}a_1^rD_{n-r-1},
 \label{eq:hessenberg-rec}\\
 P_n&=\sum_{r=0}^{m}a_{-r}a_1^rP_{n-r-1}.
 \label{eq:hessenberg-perm-rec}
\end{align}
Moreover,
\begin{align}
 \sum_{n\ge0}D_nx^n
 &=\left(1-\sum_{r=0}^{m}(-1)^r a_{-r}a_1^r x^{r+1}\right)^{-1},
 \label{eq:hessenberg-gf}\\
 \sum_{n\ge0}P_nx^n
 &=\left(1-\sum_{r=0}^{m}a_{-r}a_1^r x^{r+1}\right)^{-1}.
 \label{eq:hessenberg-perm-gf}
\end{align}
\end{proposition}

\begin{proof}
Expanding the last row, the term containing $a_{-r}$ can occur only for
$r\le n-1$; it forces, through the single available superdiagonal, a chain
of $r$ entries $a_1$ before leaving a leading principal section of order
$n-r-1$.  The determinant sign is $(-1)^r$.  Appending the terms with
$r\ge n$ does not change the sum because then $-m\le n-r-1<0$.  The
permanent expansion is the signless version.  Multiplying by $x^n$ and
summing gives the generating functions.
\end{proof}

\begin{remark}[Negative-index continuation]\label{rem:hessenberg-negative}
The restricted convention $D_{-1}=\cdots=D_{-m}=0$ (and likewise for
$P_n$) agrees with the genuine bilateral continuation of the order-$m+1$
recurrence.  The continuation does not remain zero:
\[
 D_{-(m+1)}=\frac{(-1)^m}{a_{-m}a_1^m},
 \qquad
 P_{-(m+1)}=\frac{1}{a_{-m}a_1^m}.
\]
\end{remark}

After diagonal similarity normalizes the superdiagonal $a_1$ to $1$, this
is a Toeplitz specialization of the Hessenberg generating-function method of
Getu \cite{Getu1991}.  In Section~\ref{sec:states} we recover the same
recurrence directly from the row-column state graph.

The case $m=1$ makes the Lucas connection especially explicit:
\begin{equation}\label{eq:tridiag-det-perm-Lucas}
 D_n=U_{n+1}(a_0,a_{-1}a_1),\qquad
 P_n=U_{n+1}(a_0,-a_{-1}a_1).
\end{equation}
In particular, in the one-superdiagonal family the determinant and permanent
recurrences are interchanged by the elementary substitution $a_1\mapsto-a_1$.

\section{Increasing-rows method}\label{sec:increasing}

We first give a direct elimination construction.  Its main value is
conceptual: it exposes the binomial number of boundary cofactors without
requiring the recursive state search used later in the row-column method.
We assume $m_1\ge m_2$; otherwise we transpose the matrices.

Let
\[
 d\defeq\binom{m_1+m_2}{m_1}
\]
and let $\cS$ be the collection of strictly decreasing $m_2$-tuples
\[
 S=(s_1,\ldots,s_{m_2}),
 \qquad
 2m_1-1\ge s_1>\cdots>s_{m_2}\ge m_1-m_2.
\]
Thus $|\cS|=d$.

For $\ell\ge0$ define
\[
 R_\ell=\{k+1-m_1,\ldots,k+\ell\}
\]
and
\[
 C_{\ell,S}
 =\{k-s_1,\ldots,k-s_{m_2}\}
 \cup\{k+1+m_2-m_1,\ldots,k+\ell\}.
\]
Both have cardinality $m_1+\ell$.  Let
\[
 B_{\ell,S}^{\det}\defeq\det A_{k+\ell}[R_\ell,C_{\ell,S}],
 \qquad
 B_{\ell,S}^{\perm}\defeq\perm A_{k+\ell}[R_\ell,C_{\ell,S}].
\]
Toeplitz translation invariance makes these quantities independent of $k$.

\begin{theorem}\label{thm:increasing}
For fixed $m_1,m_2$, both $(D_k(m_1,m_2))$ and $(P_k(m_1,m_2))$ satisfy a
homogeneous linear recurrence of order at most
$\binom{m_1+m_2}{m_1}$.
\end{theorem}

\begin{proof}
We treat determinants.  Expand $D_{k+\ell}$ along the last
$m_1+\ell$ rows.  Bandedness forces every nonzero complementary cofactor
to be indexed by one of the sets $S\in\cS$.  Hence, for
$\ell=0,1,\ldots,d$,
\begin{equation}\label{eq:incsystem}
 D_{k+\ell}=\sum_{S\in\cS}B_{\ell,S}^{\det}\,X_S(k),
\end{equation}
where $X_S(k)$ is the corresponding signed complementary cofactor.  The key
point is that $X_S(k)$ does not depend on $\ell$.  There are $d+1$
equations in the $d$ unknown boundary cofactors $X_S(k)$, so the coefficient
matrix has a nonzero left-kernel vector and hence yields a recurrence.  The
same argument with signless Laplace expansion proves the permanent statement.
\end{proof}

\begin{algorithm}[H]
\caption{Increasing-rows recurrence construction}\label{alg:increasing}
\begin{algorithmic}[1]
\Require $m_1\ge m_2\ge0$ and diagonal values $a_{-m_1},\ldots,a_{m_2}$
\Ensure An annihilating recurrence for $D_k$ or $P_k$
\State $d\gets\binom{m_1+m_2}{m_1}$
\State Construct the $d$ boundary states $\cS$
\For{$\ell=0,1,\ldots,d$}
  \For{$S\in\cS$}
    \State Compute $B_{\ell,S}^{\det}$ (or $B_{\ell,S}^{\perm}$)
  \EndFor
\EndFor
\State Find a nonzero vector $c=(c_0,\ldots,c_d)$ in the left kernel of $B$
\State \Return $\sum_{\ell=0}^d c_\ell X_{k+\ell}=0$, where $X=D$ or $X=P$
\end{algorithmic}
\end{algorithm}

\begin{remark}[Keeping the coefficient minors small]\label{rem:inc-centered}
The forward range $\ell=0,1,\ldots,d$ is the simplest one for the proof, but
one may instead use
\[
 -m_1\le \ell\le d-m_1,
\]
which keeps the nonzero coefficient minors at orders $0,1,\ldots,d$.  For
negative $\ell$, the column set is interpreted through the fixed
complementary cofactor.  Fix $k$ and put
\[
 H_S(k)=\{k-s_1,\ldots,k-s_{m_2}\},\qquad
 J_S(k)=\{1,\ldots,k+m_2-m_1\}\setminus H_S(k).
\]
If $-m_1\le\ell<0$, set the coefficient to zero unless
$J_S(k)\subseteq\{1,\ldots,k+\ell\}$; otherwise use
\[
 R_\ell=\{k+1-m_1,\ldots,k+\ell\},\qquad
 C_{\ell,S}(k)=\{1,\ldots,k+\ell\}\setminus J_S(k).
\]
For example, when $(m_1,m_2)=(2,1)$ one may use
$\ell=-2,-1,0,1$, so only minors of orders at most $3$ occur.
\end{remark}

Algorithm~\ref{alg:increasing} is easy to state but becomes expensive because
its entries are themselves determinants or permanents of increasingly large
symbolic minors.  The row-column method avoids this growth by recursively
closing a finite family of boundary minors.

\subsection{Complementary cofactors as a compound state}\label{sec:inc-compound}

The exterior-power mechanism itself is classical.  In the determinant part
of Shevelev's algomatrix construction, the relevant algomatrix is identified
with an associated (compound) matrix of a companion-type matrix, and its
eigenvalues are the corresponding fixed-cardinality products of the roots
\cite{She}.  Shevelev's indexing uses the exterior degree complementary to
the convention adopted below; the two descriptions have the same binomial
dimension and are related by the usual complementary-compound duality.  What
we need here is a direct identification of the particular complementary
cofactors produced by Algorithm~\ref{alg:increasing} with this classical
compound state.  Put
\[
 d_0=m_1+m_2,\qquad p=m_2,
\]
and for a $p$-subset $I\subseteq[d_0]$ define
\begin{equation}\label{eq:DeltaI}
 \Delta_I(n)\defeq
 \det A_{n+p}\!\left[
   \{1,\ldots,n\},
   \{1,\ldots,n+p\}\setminus
   \{n-m_1+i:i\in I\}
 \right].
\end{equation}
For the base subset
$B_0=\{m_1+1,\ldots,d_0\}$ we have $\Delta_{B_0}(n)=D_n$.
Let
\[
 Y_I(n)\defeq
 (-1)^{\sum_{i\in I}i-\sum_{i\in B_0}i}\Delta_I(n).
\]
Let $C_q$ be the Frobenius companion matrix of the monic polynomial
$q(z)/a_{m_2}$, in the convention
$C_qe_i=e_{i+1}$ for $1\le i<d_0$.

\begin{theorem}[Cofactor realization of the compound transfer]\label{thm:compound-transfer}
With the coordinates ordered by the $p$-subsets of $[d_0]$,
\begin{equation}\label{eq:compound-transfer}
 Y(n+1)=T\,Y(n),\qquad
 T=(-1)^p a_{m_2}\,\bigwedge^p C_q.
\end{equation}
Moreover, the complementary cofactors $X_S(k)$ in the proof of
Theorem~\ref{thm:increasing} are, up to their standard cofactor signs and a
uniform shift of $n$, exactly the coordinates $\Delta_I(n)$.
\end{theorem}

\begin{proof}
Expand \eqref{eq:DeltaI} along its last row.  If $d_0\notin I$, the last
column of the ambient rectangular section is present.  All pivots except
that last column leave an identically zero last column after deletion, so
there is one transition, obtained by shifting the hole set.  If $d_0\in I$,
the surviving pivots replace the companion coordinate by one of the allowed
band positions.  The coefficients are exactly those of the last column of
$C_q$.  The displayed sign gauge supplies the wedge orientation, giving
\eqref{eq:compound-transfer}.

For $S=(s_1,\ldots,s_p)\in\cS$ put
\[
 I(S)=\{2m_1-s:s\in S\}.
\]
The defining range of $S$ makes $I(S)$ a $p$-subset of $[d_0]$, and every
$p$-subset occurs exactly once.  Comparing the retained and deleted columns
in \eqref{eq:incsystem} with \eqref{eq:DeltaI} identifies its complementary
cofactor with $\Delta_{I(S)}$ after the shift $n=k-m_1$.
\end{proof}

Thus the binomial scale in the increasing-rows method is not accidental:
Algorithm~\ref{alg:increasing} realizes the classical compound/associated-
matrix mechanism in concrete Laplace-cofactor coordinates.  The point of
Theorem~\ref{thm:compound-transfer} is this cofactor-level identification,
which allows a direct comparison with the sparse row-column states below.

\section{Row-column transfer method}\label{sec:rowcolumn}

\subsection{Pentadiagonal prototype}

We begin with $(m_1,m_2)=(2,2)$ and display every Laplace expansion needed
to reveal the six-state graph.  Some of these relations are related by
transposition, which interchanges $a_j$ and $a_{-j}$, but one of the
self-dual boundary minors has to be expanded separately; writing the full
system makes the origin of every row of the transfer matrix visible.

For integers $k>n\ge0$ and increasing index sets, let
\[
 \widehat M_{k,k-j_1,\ldots,k-j_n}^{k,k-i_1,\ldots,k-i_n}
\]
denote the determinant of the minor of $A_{k-1}$ obtained by deleting the
rows $k-i_1,\ldots,k-i_n$ and columns $k-j_1,\ldots,k-j_n$.  We call
\[
 \sigma=\begin{bmatrix}i_1,\ldots,i_n\\ j_1,\ldots,j_n\end{bmatrix}
\]
its \emph{signature}.  For $n=0$ we have $\widehat M_k^k=D_{k-1}$.

Expanding the last available row of the empty-signature minor gives
\begin{equation}\label{eq:rc1}
 \widehat M_{k+1}^{k+1}
 =a_0\widehat M_k^k
 -a_{-1}\widehat M^{k+1,k}_{k+1,k-1}
 +a_{-2}\widehat M^{k+1,k}_{k+1,k-2}.
\end{equation}
The two new minors are then expanded at their exposed boundary columns:
\begin{align}
\widehat M^{k+1,k}_{k+1,k-1}
 &=a_1\widehat M^{k,k-1}_{k,k-1}
   -a_2\widehat M^{k,k-2}_{k,k-1}, \label{eq:rc2}\\
\widehat M^{k+1,k}_{k+1,k-2}
 &=a_1\widehat M^{k,k-1}_{k,k-2}
   -a_2\widehat M^{k,k-2}_{k,k-2}. \label{eq:rc3}
\end{align}
The remaining two one-deletion signatures must also be expanded explicitly:
\begin{align}
\widehat M^{k,k-2}_{k,k-1}
 &=a_{-1}\widehat M^{k-1,k-2}_{k-1,k-2}
   -a_{-2}\widehat M^{k-1,k-2}_{k-1,k-3}, \label{eq:rc4}\\
\widehat M^{k,k-2}_{k,k-2}
 &=a_0\widehat M^{k-1,k-2}_{k-1,k-2}
   -a_{-2}\widehat M^{k,k-1,k-2}_{k,k-2,k-3}. \label{eq:rc5}
\end{align}
The first is the transpose-dual of an earlier boundary pattern after
interchanging $a_j$ and $a_{-j}$; the second is self-dual and is not supplied
by transposition.  Finally,
\begin{equation}\label{eq:rc6}
 \widehat M^{k,k-1,k-2}_{k,k-2,k-3}=a_2\widehat M_{k-3}^{k-3}.
\end{equation}
Thus all six Laplace expansions are visible and close the process on six
signatures.

With
\[
 v_k^{\mathsf T}=\bigl[
 \widehat M_k^k,
 \widehat M^{k+1,k}_{k+1,k-1},
 \widehat M^{k+1,k}_{k+1,k-2},
 \widehat M^{k+1,k-1}_{k+1,k},
 \widehat M^{k+1,k-1}_{k+1,k-1},
 \widehat M^{k+2,k+1,k}_{k+2,k,k-1}
 \bigr],
\]
we obtain
\begin{equation}\label{eq:Pdet}
 v_{k+1}=Qv_k,
 \qquad
 Q=
 \begin{bmatrix}
 a_0&-a_{-1}&a_{-2}&0&0&0\\
 a_1&0&0&-a_2&0&0\\
 0&a_1&0&0&-a_2&0\\
 a_{-1}&-a_{-2}&0&0&0&0\\
 a_0&0&0&0&0&-a_{-2}\\
 a_2&0&0&0&0&0
 \end{bmatrix}.
\end{equation}
By Cayley--Hamilton, the characteristic polynomial of $Q$ annihilates every
coordinate sequence of $v_k$, in particular $D_{k-1}$, and gives
\eqref{eq:sweet}.

For a recurrence of a distinguished coordinate, however, it is not
necessary to compute the full characteristic polynomial of $Q$.  Standard
Krylov scalarization, applied to $Q^{\mathsf T}$ and the corresponding
coordinate vector $e$, finds the first linear dependence among
$e^{\mathsf T},e^{\mathsf T}Q,e^{\mathsf T}Q^2,\ldots$ and hence the vector
minimal polynomial of $e$ for $Q^{\mathsf T}$, equivalently of
$e^{\mathsf T}$ under right multiplication by $Q$; see
\cite[Chapter~8]{ZimmermannEtAl2018}.  This polynomial annihilates the
observable quotient and can have degree smaller than the transfer dimension.
It divides the matrix minimal polynomial of $Q$, which in turn divides the
characteristic polynomial; for a particular initial state the resulting
scalar sequence can satisfy an even smaller minimal recurrence.  The gap
between a mechanically generated transfer dimension and a much smaller
annihilator has an earlier circulant-permanent precedent: in the notation of
Golin, Leung, and Wang \cite{GolinLeungWang2006}, their two-sided cycle-cover
classification gives a $2^{2\bar s}$-dimensional transfer, while a polynomial
of degree $2^{\bar s}-1$ already annihilates it and recovers Minc's
recurrence.  That result is an annihilation statement rather than a generic
minimality theorem.  For special parameter subfamilies the observable degree can be smaller than
the transfer dimension.  This is one reason to separate construction of an
annihilator from subsequent scalar minimality analysis.  Implementations of
the Krylov scalarization are included in the accompanying code repository.

The next balanced case is displayed in Appendix~\ref{app:33}: for
$(m_1,m_2)=(3,3)$ the construction closes on $20$ signatures and gives a
$20\times20$ transfer matrix with $50$ nonzero entries.

For permanents the same six signatures occur, but the Laplace signs are
removed.  Thus
\begin{equation}\label{eq:Pperm}
 \widetilde v_{k+1}=Q^+\widetilde v_k,
 \qquad
 Q^+=
 \begin{bmatrix}
 a_0&a_{-1}&a_{-2}&0&0&0\\
 a_1&0&0&a_2&0&0\\
 0&a_1&0&0&a_2&0\\
 a_{-1}&a_{-2}&0&0&0&0\\
 a_0&0&0&0&0&a_{-2}\\
 a_2&0&0&0&0&0
 \end{bmatrix}.
\end{equation}
Its characteristic polynomial gives
\begin{align}
P_{k+6}={}&a_0P_{k+5}+(a_{-1}a_1+a_{-2}a_2)P_{k+4}
 +(a_{-2}a_1^2+a_{-1}^2a_2)P_{k+3}\notag\\
&+a_{-2}a_2(a_{-1}a_1+a_{-2}a_2)P_{k+2}
 -a_0(a_{-2}a_2)^2P_{k+1}-(a_{-2}a_2)^3P_k.
\label{eq:pentaperm}
\end{align}
\subsection{Formal row-column algorithm}\label{sec:algorithm2}

A \emph{boundary minor family} is represented by a normalized signature
$\sigma=(I,J)$ recording the rows and columns missing near the lower-right
boundary of the growing Toeplitz section.  Two representatives are
identified if one is obtained from the other by simultaneously translating
all row and column indices; common terminal rows and columns are stripped
before the signature is stored.

For a signature $\sigma$, let $M_\sigma(k)$ denote its normalized $k$th
representative.  To expand $M_\sigma(k+1)$ we choose the remaining row or
column with the largest original index; in case of a tie we choose the row.
Each nonzero Laplace term is normalized to another signature and contributes
a coefficient $\pm a_s$.  Distinct nonzero terms have distinct normalized
target signatures.  For permanents the same transition graph is used with
the signs removed.  For definiteness, we number signatures in
\emph{first-discovery order}, beginning with the principal state
$\varnothing$; this is the ordering used in the explicit transfer matrices
below and in Appendix~\ref{app:33}.

\begin{algorithm}[H]
\caption{Row-column transfer construction}\label{alg:rowcolumn}
\begin{algorithmic}[1]
\Require $m_1,m_2$ and diagonal values $a_{-m_1},\ldots,a_{m_2}$
\Ensure A finite transfer matrix $Q$ whose characteristic polynomial annihilates $D_k$ (or $P_k$)
\State $\mathsf{Todo}\gets[\varnothing]$ (an ordered queue), $\mathsf{Done}\gets\varnothing$
\State Assign state number $1$ to $\varnothing$
\While{$\mathsf{Todo}\ne\varnothing$}
  \State Remove the first signature $\sigma$ from $\mathsf{Todo}$
  \State Choose the largest-index available boundary row or column of $M_\sigma(k+1)$
  \State Laplace-expand along it
  \For{each nonzero term $c\,M(k)$ in the expansion, in Laplace order}
     \State Normalize $M(k)$ to a signature $\tau$
     \State Record $q_{\sigma,\tau}\gets c$
     \If{$\tau\notin\mathsf{Done}\cup\mathsf{Todo}$}
        \State Append $\tau$ to $\mathsf{Todo}$ and assign it the next state number
     \EndIf
  \EndFor
  \State Add $\sigma$ to $\mathsf{Done}$
\EndWhile
\State Assemble $Q$ in this first-discovery order
\State \Return $Q$
\end{algorithmic}
\end{algorithm}

For fixed $m_1,m_2$, Algorithm~\ref{alg:rowcolumn} terminates.  Indeed, let
$w=m_1+m_2$.  At every step the pivot is an extreme available row or column,
and a nonzero entry in that pivot can connect only to an index within the
fixed band.  After normalization, the discrepancy between a boundary minor
and a leading principal Toeplitz section is therefore confined to a boundary
window of width depending only on $w$; no expansion can create a deleted row
or column arbitrarily far behind the moving boundary.  Hence only finitely
many normalized signatures can occur, so the queue eventually empties.  The
next subsection sharpens this coarse finiteness argument by describing the
reachable signatures exactly and counting the resulting transitions.

\subsection{Exact signature classification and sparsity}\label{sec:states}

We now determine the normalized signature graph exactly.  We assume here that
$m_1,m_2\ge1$; the one-sided edge cases $m_1=0$ or $m_2=0$ are triangular and immediate.
The statement concerns the symbolic band coefficients, before a specialization
can make individual transitions vanish.  It is convenient to use the
\emph{final normalized representative}: at level $j$ its length
is $\ell=j+1$ and, after removing the common moving $k$-offset, it has the
form
\[
 R=(\ell,R_2,\ldots,R_\ell),\qquad
 C=(\ell,C_2,\ldots,C_\ell).
\]
This is precisely the form immediately before the Laplace expansion in
Algorithm~\ref{alg:rowcolumn}.  Let $m=m_1$ and $p=m_2$.  For an integer
$r$, write $\langle r\rangle_m$ for its representative in $[m]$ modulo $m$.

For $A\in\binom{[m]}j$ and $B\in\binom{[p]}j$, define
\[
 \widehat A=
 \begin{cases}
 A,&1\notin B,\\
 (A\setminus\{1\})\cup
   \bigl(\{m+1\}\text{ if }1\in A\bigr),&1\in B,
 \end{cases}
\]
and, writing $B=\{b_1<\cdots<b_j\}$ and
$\widehat A=\{\widehat a_1<\cdots<\widehat a_j\}$, define
\begin{equation}\label{eq:sigmaAB}
 \Sigma(A,B)=
 \left(
  (\ell,\ell-b_1,\ldots,\ell-b_j),
  (\ell,\ell-\widehat a_1,\ldots,\ell-\widehat a_j)
 \right).
\end{equation}
For $j=0$, this is the principal state $((1),(1))$.

\begin{theorem}[Exact row-column state structure]\label{thm:states}
Algorithm~\ref{alg:rowcolumn}, for determinants or permanents, reaches
exactly the states \eqref{eq:sigmaAB}, with
\[
 A\in\binom{[m_1]}j,\qquad B\in\binom{[m_2]}j,
 \qquad0\le j\le\min(m_1,m_2).
\]
Consequently the number of states at level $j$ is
\begin{equation}\label{eq:statelevel}
 N_j=\binom{m_1}{j}\binom{m_2}{j},
\end{equation}
the total transfer dimension is
\begin{equation}\label{eq:statecount}
 \sum_j\binom{m_1}{j}\binom{m_2}{j}
 =\binom{m_1+m_2}{m_1},
\end{equation}
and the number of nonzero transition entries is
\begin{equation}\label{eq:sparsity}
 N_{\mathrm{nz}}
 =\frac{m_1m_2+m_1+m_2}{m_1+m_2}
  \binom{m_1+m_2}{m_1}.
\end{equation}
\end{theorem}

\begin{proof}
Let a final normalized representative be $(R,C)$ of length $\ell=j+1$.
If $x_1<\cdots<x_j$ and $y_1<\cdots<y_j$ are the increasing rearrangements
of $R_2,\ldots,R_\ell$ and $C_2,\ldots,C_\ell$, respectively, define
\begin{equation}\label{eq:AB-from-signature}
 B_\sigma=\{\ell-x_t:1\le t\le j\},\qquad
 A_\sigma=\{\langle\ell-y_t\rangle_m:1\le t\le j\}.
\end{equation}
The formula \eqref{eq:sigmaAB} is inverse to this map.

The normalization rule has a simple deficit-coordinate interpretation:
whenever the row and column deficit sets both contain $1$, remove the common
$1$ and subtract $1$ from every remaining deficit.  In these coordinates
$1\in B$ is equivalent to $R_1-R_2=1$, hence tells us exactly whether the
algorithm expands a column or a row.  Direct substitution of
\eqref{eq:sigmaAB} into the two Laplace rules shows that every child is again
of the form $\Sigma(A',B')$.  Thus the displayed family is closed under the
algorithm.

For the converse, let $j>0$.  If $1\in B$, choose
$a_*=1$ when $1\in A$ and otherwise put $a_*=\min A$.  Put
\[
 A_0\defeq A\setminus\{a_*\},\qquad
 B_0\defeq B\setminus\{1\},
\]
and use the row pivot
\[
 q=\begin{cases}m,&a_*=1,\\a_*-1,&a_*>1.\end{cases}
\]
in $\Sigma(A_0,B_0)$.  Its normalized child is exactly $\Sigma(A,B)$.

If $1\notin B$, choose $b_*=p$ when $p\in B$ and otherwise take
$b_*=\max B$.  Define
\[
 \begin{aligned}
 B_1&\defeq\{1\}\cup\{b+1: b\in B\setminus\{b_*\}\},\\
 A_1&\defeq\{\langle a+1\rangle_m: a\in A\}.
 \end{aligned}
\]
Then $1\in B_1$, so the preceding paragraph makes $\Sigma(A_1,B_1)$
reachable, and its column expansion with pivot $q=b_*$ has normalized child
$\Sigma(A,B)$.  Induction on $j$ proves reachability of every displayed
state.

The level and total counts follow immediately, the latter from Vandermonde's
identity.  A state with $1\notin B$ is a row-expansion state and has exactly
$m+1-j$ nonzero Laplace terms; a state with $1\in B$ is a column-expansion
state and has exactly $p+1-j$ terms.  There are respectively
\[
 \binom mj\binom{p-1}{j},\qquad
 \binom mj\binom{p-1}{j-1}
\]
such states.  Summing these contributions over $j$ and using Vandermonde
gives \eqref{eq:sparsity}.  Laplace signs do not affect the graph, so the
same classification holds for permanents.
\end{proof}

For $(m_1,m_2)=(3,3)$, Theorem~\ref{thm:states} gives the $20$ states and
$50$ entries displayed explicitly in Appendix~\ref{app:33}.

\begin{corollary}[The $m_2=1$ state graph]\label{prop:states-m1}
For the symbolic $(m,1)$ family, Algorithm~\ref{alg:rowcolumn} reaches
exactly $m+1$ states.  In the order
\[
 \varnothing,\quad
 \sigma_1=\begin{bmatrix}1\\2\end{bmatrix},\quad
 \sigma_2=\begin{bmatrix}1\\3\end{bmatrix},\quad\ldots,\quad
 \sigma_m=\begin{bmatrix}1\\m+1\end{bmatrix},
\]
the determinant transfer matrix is
\begin{equation}\label{eq:Qm1}
 Q_{m,1}=
 \begin{bmatrix}
 a_0&-a_{-1}&a_{-2}&\cdots&(-1)^m a_{-m}\\
 a_1&0&0&\cdots&0\\
 0&a_1&0&\cdots&0\\
 \vdots&&\ddots&\ddots&\vdots\\
 0&\cdots&0&a_1&0
 \end{bmatrix}.
\end{equation}
Its characteristic polynomial is the denominator of
\eqref{eq:hessenberg-gf} with $x=t^{-1}$ and powers of $t$ cleared.
\end{corollary}

\begin{remark}[From a recurrence to a distant term]\label{rem:fastterms}
Once an annihilating recurrence is known, a distant term need not be
generated sequentially.  Fiduccia's method \cite{Fid}, the modular squaring
approach of Khomovsky \cite{Kho}, and the Bostan--Mori algorithm
\cite{BostanMori2021} evaluate distant terms efficiently.  Thus recurrence
construction and recurrence evaluation are naturally separate stages.
\end{remark}

\subsection{Comparison with the compound companion}\label{sec:compound-comparison}

Theorem~\ref{thm:states} determines the row-column state graph, while
Theorem~\ref{thm:compound-transfer} realizes the complementary cofactors of
the increasing-rows construction in the classical compound companion state,
consistent with Shevelev's associated-matrix construction \cite{She}.  We now
compare the two determinant transfers directly.  Put
\[
 m=m_1,\qquad p=m_2,\qquad d=\binom{m+p}{p},
\]
and work over the rational function field
$F=\mathbb Q(a_{-m},\ldots,a_p)$.  Let
\[
 T=(-1)^pa_p\,\bigwedge^p C_q,
\]
so that the characteristic polynomial of $T$ is the Widom polynomial
$\chi_{m,p}(t)$ from \eqref{eq:widompoly}.

\begin{theorem}[Generic row-column/compound similarity]\label{thm:rowcolumn-compound}
Let $Q$ be the determinant row-column transfer matrix of
Algorithm~\ref{alg:rowcolumn}.  Then, over $F$,
\begin{equation}\label{eq:charQ-Widom}
 \det(tI-Q)=\chi_{m,p}(t).
\end{equation}
Moreover, $Q$ and $T$ are similar over $F$.  Hence the two transfers are
similar after every specialization outside a proper algebraic subset of the
genuine-band parameter space.
\end{theorem}

\begin{proof}
By Theorem~\ref{thm:states}, $Q$ has size $d$.  Its characteristic
polynomial is therefore a monic degree-$d$ annihilator of the principal
determinant coordinate.  Widom's formula \eqref{eq:widomform} gives another
monic degree-$d$ annihilator, namely $\chi_{m,p}(t)$.

On the nonempty Zariski-open set where the roots of $q$ are simple, the
products $W_J$ in \eqref{eq:widomform} are pairwise distinct, and the
corresponding Widom coefficients are nonzero.  There the determinant
sequence has minimal scalar recurrence order $d$.  Hence its two monic
degree-$d$ annihilators coincide.  Since their coefficients are rational
functions of the band parameters, equality on this dense open set proves
\eqref{eq:charQ-Widom} over $F$.

The polynomial $\chi_{m,p}(t)$ is squarefree over $F$, because its roots are
generically distinct.  Thus both $Q$ and $T$ have minimal polynomial equal
to their common characteristic polynomial $\chi_{m,p}(t)$.  Each is
therefore cyclic, with rational canonical form consisting of the single
companion block of $\chi_{m,p}$, and hence $Q\sim T$ over $F$.
\end{proof}

\begin{remark}[Exterior-power meaning of the levels]\label{rem:plucker-states}
The $p$-subsets of $[m+p]$ are equivalently partitions inside the
$p\times m$ rectangle.  The exchange level $j$ is their Durfee/Frobenius
rank, and
\[
 \#\{\lambda:\operatorname{Durfee}(\lambda)=j\}
 =\binom mj\binom pj,
\]
which is exactly the level count in Theorem~\ref{thm:states}.
\end{remark}

\subsection{Comparison of Laplace-closure state models}\label{sec:laplace-comparison}
The central binomial scale in the balanced determinant problem is visible in
several constructive descriptions, but the coordinates and closure mechanisms
are genuinely different.  The increasing-rows method eliminates a fixed
family of complementary cofactors; the row-column method recursively closes
translated boundary minors; the compound/Widom description uses a fixed
exterior power; and Hendel's independent Laplace procedure
\cite{Hendel2026} proceeds by successive first-row expansions followed by a
closing first-column expansion.

If $R$ is Hendel's Toeplitz order, Lemma~7.1 of \cite{Hendel2026} indexes the
boundary-column patterns at expansion stage $k$ by
\[
 \{(b_1,\ldots,b_k):2\le b_1<\cdots<b_k\le R+k-1\},
\]
so there are $\binom{R+k-2}{k}$ such patterns.  For a balanced $(m,m)$ band,
$R=m+1$, and the initial determinant family together with stages
$1,\ldots,m$ therefore has total size
\[
 1+\sum_{k=1}^{m}\binom{m+k-1}{k}=\binom{2m}{m}.
\]
Thus all four viewpoints meet the same binomial dimension in the balanced
case, without implying that their states are identical.

\begin{center}
\footnotesize
\setlength{\tabcolsep}{3.5pt}
\begin{tabular}{@{}>{\raggedright\arraybackslash}p{0.14\textwidth}>{\raggedright\arraybackslash}p{0.24\textwidth}>{\raggedright\arraybackslash}p{0.32\textwidth}>{\centering\arraybackslash}p{0.12\textwidth}@{}}
\toprule
\textbf{Realization} & \textbf{State coordinates} & \textbf{Closure/evolution} & \textbf{Balanced size} \\
\midrule
Hendel & boundary-column sets after successive leading-row deletions & staged first-row expansion with closing first-column step & $\binom{2m}{m}$ \\
Increasing rows & fixed complementary cofactors & simultaneous elimination across $d+1$ growing sections & $\binom{2m}{m}$ \\
Row-column & exchanged boundary-subset pairs $(A,B)$ & recursive extreme-boundary Laplace expansion and normalization & $\binom{2m}{m}$ \\
Compound / Widom & $m$-subsets, equivalently $\bigwedge^m$ coordinates & companion/exterior-power action and root-product modes & $\binom{2m}{m}$ \\
\bottomrule
\end{tabular}
\end{center}

Despite their common binomial scale, the Laplace closures discussed above
use genuinely different state stratifications.  Theorem~\ref{thm:compound-transfer}
identifies the increasing-rows cofactors with the compound coordinates, while
Theorem~\ref{thm:rowcolumn-compound} shows that the independently generated
row-column transfer is generically similar to the same compound realization.
The agreement of the counts is therefore structural evidence for the binomial
state size, not a consequence of using the same state model.

\section{Sparse Toeplitz--Hessenberg spectral specialization}\label{sec:sparse-spectrum}

The recurrence theory above can also encode spectral information for concrete
Toeplitz families.  We record one particularly sparse example because its
scalar recurrence collapses to a classical threefold-symmetric
$2$-orthogonal polynomial system.  The root-of-unity geometry of positive
double-band matrices is itself classical: McMillen \cite{McMillen2009}
proved that their nonzero eigenvalues occur in root-of-unity orbits whose
period is determined by the separation of the two bands; for offsets $-2$
and $+1$ the period is $3$.  Our purpose here is therefore not to claim the
three-ray geometry in isolation, but to combine the explicit scalar
recurrence with the $2$-orthogonal identification, which makes the radial
roots, their interlacing, and the sharp endpoint transparent.  This section
is logically independent of the two general recurrence constructions: it
illustrates what can be read off once a short Toeplitz--Hessenberg recurrence
has been identified.

\begin{corollary}[A three-star spectrum from the sparse $(2,1)$ family]
\label{cor:sparse-hessenberg-spectrum}
Let $A_n=A_n(x,y)$ be the $n\times n$ Toeplitz matrix whose only possibly
nonzero diagonals are $a_{-2}=y$ and $a_1=x$, and put
$c\defeq x^2y$.  If
\[
 \chi_n(\lambda)\defeq\det(\lambda I_n-A_n),
\]
then
\begin{equation}\label{eq:sparse-hessenberg-rec}
 \chi_n(\lambda)=\lambda\chi_{n-1}(\lambda)-c\chi_{n-3}(\lambda),
 \qquad
 \chi_0=1,\quad \chi_1=\lambda,\quad \chi_2=\lambda^2,
\end{equation}
and consequently
\begin{align}
 \sum_{n\ge0}\chi_n(\lambda)t^n
 &=\frac{1}{1-\lambda t+ct^3},
 \label{eq:sparse-hessenberg-gf}\\
 \chi_n(\lambda)
 &=\sum_{j=0}^{\lfloor n/3\rfloor}
   (-1)^j\binom{n-2j}{j}c^j\lambda^{n-3j}.
 \label{eq:sparse-hessenberg-explicit}
\end{align}
For $r\in\{0,1,2\}$ define
\begin{equation}\label{eq:sparse-cubic-component}
 R_m^{(r)}(u)
 \defeq
 \sum_{j=0}^{m}(-1)^j
 \binom{3m+r-2j}{j}u^{m-j}.
\end{equation}
If $c\ne0$, then
\begin{equation}\label{eq:sparse-cubic-decomposition}
 \chi_{3m+r}(\lambda)
 =c^m\lambda^r
 R_m^{(r)}\!\left(\frac{\lambda^3}{c}\right).
\end{equation}
Each $R_m^{(r)}$ has $m$ simple positive zeros, denoted
$0<\xi_{m,1}^{(r)}<\cdots<\xi_{m,m}^{(r)}$, and they satisfy
\begin{equation}\label{eq:sparse-cubic-interlacing}
 0<\xi_{m,1}^{(0)}<\xi_{m,1}^{(1)}<\xi_{m,1}^{(2)}
 <\xi_{m,2}^{(0)}<\cdots<\xi_{m,m}^{(2)}<\frac{27}{4}.
\end{equation}
For each fixed $r\in\{0,1,2\}$, consecutive cubic components also strictly
interlace:
\begin{equation}\label{eq:sparse-fixed-r-interlacing}
 0<\xi_{m+1,1}^{(r)}<\xi_{m,1}^{(r)}<\xi_{m+1,2}^{(r)}
 <\cdots<\xi_{m,m}^{(r)}<\xi_{m+1,m+1}^{(r)}<\frac{27}{4}.
\end{equation}
If $c=0$, then $\chi_n(\lambda)=\lambda^n$ and $A_n$ is nilpotent.  If
$c\ne0$, then, with $\omega=e^{2\pi i/3}$, zero has algebraic multiplicity
$r$ in $A_{3m+r}$ and
\begin{equation}\label{eq:sparse-hessenberg-spectrum}
 \operatorname{Spec}(A_{3m+r})\setminus\{0\}
 =\left\{
 (c\xi_{m,k}^{(r)})^{1/3}\omega^j:
 1\le k\le m,\ 0\le j\le2
 \right\},
\end{equation}
where any one choice of the cubic root is used in each triple.  In particular,
for $c>0$ the nonzero spectrum lies on the three rays of arguments
$0,2\pi/3,4\pi/3$, and the radial spectra for the three orders
$3m,3m+1,3m+2$ strictly interlace.  For arbitrary complex $c\ne0$ the same
three-star is rotated by $\arg(c)/3$, and every eigenvalue satisfies
\begin{equation}\label{eq:sparse-hessenberg-radius}
 |\lambda|<\left(\frac{27|c|}{4}\right)^{1/3}.
\end{equation}
Moreover, for each fixed $r$,
\begin{equation}\label{eq:sparse-largest-zero-limit}
 \lim_{m\to\infty}\xi_{m,m}^{(r)}=\frac{27}{4},
\end{equation}
and hence, for fixed $c$,
\begin{equation}\label{eq:sparse-spectral-radius-limit}
 \lim_{m\to\infty}
 \max\{|\lambda|:\lambda\in\operatorname{Spec}(A_{3m+r})\}
 =\left(\frac{27|c|}{4}\right)^{1/3}.
\end{equation}
Thus the radius in \eqref{eq:sparse-hessenberg-radius} is asymptotically
sharp.
\end{corollary}

\begin{proof}
Apply Proposition~\ref{prop:hessenberg} to $\lambda I_n-A_n$: its main
diagonal is $\lambda$, its first superdiagonal is $-x$, and its second
subdiagonal is $-y$.  This gives \eqref{eq:sparse-hessenberg-rec}; summation
and coefficient extraction give \eqref{eq:sparse-hessenberg-gf} and
\eqref{eq:sparse-hessenberg-explicit}.  Grouping powers of $\lambda$ modulo
$3$ gives \eqref{eq:sparse-cubic-decomposition}.

After the rescaling $\lambda=c^{1/3}z$, the normalized polynomials satisfy
\[
 C_0=1,\qquad C_1=z,\qquad C_2=z^2,\qquad
 C_{n+1}=zC_n-C_{n-2}.
\]
Up to the standard scalar normalization, this is the constant-coefficient
threefold-symmetric $2$-orthogonal Tchebychev sequence of Douak and Maroni
\cite{DouakMaroni1997I}.  The zero
theorem for threefold-symmetric $2$-orthogonal sequences with positive
recurrence coefficients gives positivity, simplicity, the cross-component
interlacing in \eqref{eq:sparse-cubic-interlacing}, and the successive-component
interlacing in \eqref{eq:sparse-fixed-r-interlacing}; see
\cite{BenRomdhane2008,LoureiroVanAssche2020}.  If $c=0$, the recurrence
\eqref{eq:sparse-hessenberg-rec} immediately gives $\chi_n(\lambda)=\lambda^n$.
For $c\ne0$, Equation \eqref{eq:sparse-hessenberg-spectrum} follows directly
from \eqref{eq:sparse-cubic-decomposition}.

For the finite-size radius bound one may argue directly.  Let
$D_s=\operatorname{diag}(1,s,\ldots,s^{n-1})$, $s>0$.  Every row sum of
$D_s^{-1}|A_n|D_s$ is at most
\[
 |x|s+|y|s^{-2}.
\]
The minimum occurs at $s^3=2|y|/|x|$ and equals
$3|x^2y|^{1/3}/2^{2/3}$.  When $xy\ne0$ and $n\ge3$, the nonnegative support
matrix is irreducible and at least one boundary row has a strict inequality,
so Perron--Frobenius makes the spectral-radius bound strict.  Together with
\eqref{eq:sparse-hessenberg-spectrum}, this gives
$\xi_{m,k}^{(r)}<27/4$ and \eqref{eq:sparse-hessenberg-radius}.  The cases
$n<3$ are immediate.

For the constant-coefficient $2$-orthogonal Tchebychev family, the normalized
zero counting measures of the cubic components have a limiting distribution
supported on $0<t<1$ after the scaling $t=4u/27$, with positive density on
that interval; see \cite{LoureiroVanAssche2020,ShapiroStampach2019}.  Hence
for every $\varepsilon>0$ the interval $(1-\varepsilon,1)$ contains zeros for
all sufficiently large degrees, while the finite-size bound keeps all scaled
zeros below $1$.  This proves \eqref{eq:sparse-largest-zero-limit}; taking
cubic roots in \eqref{eq:sparse-hessenberg-spectrum} gives
\eqref{eq:sparse-spectral-radius-limit}.
\end{proof}

\begin{remark}\label{rem:sparse-hessenberg-literature}
The $2$-orthogonal-polynomial theory is used here only for the zero geometry.
The recurrence, rational generating function, coefficient formula, and cubic
decomposition are direct consequences of the Toeplitz--Hessenberg recurrence.
For positive $x,y$, the threefold rotational geometry is already a special
case of McMillen's double-band theorem \cite{McMillen2009}; the present
formulation identifies the corresponding radii with the zeros of the three
cubic components and extends the algebraic spectral formula to arbitrary
complex $c=x^2y$.  Ben Romdhane \cite{BenRomdhane2008} also studies
eigenvalues of the associated banded Hessenberg matrices.  An explicit limiting density on the normalized
interval $0<t<1$ is known for the constant-coefficient case; see
\cite{ShapiroStampach2019}.  We use only its endpoint consequence
\eqref{eq:sparse-largest-zero-limit}, rather than reproducing the density
formula.  We therefore regard Corollary~\ref{cor:sparse-hessenberg-spectrum}
as a translation of known zero theory into the notation of this Toeplitz
specialization, rather than as a new theory of the underlying polynomials.
\end{remark}

\begin{remark}[Classical $2$-tridiagonal comparison]
\label{rem:sparse-parity-split}
For comparison, let $B_n$ be the $n\times n$ Toeplitz matrix whose only
possibly nonzero diagonals are $b_{-2}=y$ and $b_2=x$.  This is the
$k=2$ instance of the classical permutation decomposition for
$k$-tridiagonal matrices, going back to Egerv{\'a}ry and Sz{\'a}sz
\cite{EgervarySzasz1928}; see also Losonczi \cite{Losonczi1992}, the unified
account of da Fonseca and Y{\i}lmaz \cite{DaFonsecaYilmaz2015}, and the
pentadiagonal discussion in An\v{d}eli'c and da Fonseca \cite{An}.  If
$p=\lceil n/2\rceil$ and $q=\lfloor n/2\rfloor$, simultaneous permutation of
rows and columns according to odd and even indices gives
\[
 B_n\sim C_p\oplus C_q,
 \qquad
 C_s=
 \begin{bmatrix}
  0&x&&\\
  y&0&\ddots&\\
   &\ddots&\ddots&x\\
   &&y&0
 \end{bmatrix}_{s\times s}.
\]
Thus the tridiagonal Lucas formula from Section~\ref{sec:tridiag} gives
\begin{equation}\label{eq:sparse-parity-charpoly}
 \det(\lambda I_n-B_n)
 =U_{p+1}(\lambda,xy)\,U_{q+1}(\lambda,xy).
\end{equation}
Equivalently, with either choice of $\sqrt{xy}$,
\[
 \operatorname{Spec}(B_n)=
 \left\{2\sqrt{xy}\cos\frac{\pi k}{p+1}:1\le k\le p\right\}
 \cup
 \left\{2\sqrt{xy}\cos\frac{\pi k}{q+1}:1\le k\le q\right\},
\]
where the union is a multiset union.  In particular, when $n=2p$ and
$xy\ne0$, every eigenvalue occurs with multiplicity two.  This elementary
parity decomposition is useful for distinguishing the offset-$\pm2$ support
from the genuinely threefold-symmetric $(2,1)$ family above.
\end{remark}

\section{Beyond autonomous open Toeplitz sections}\label{sec:variable}

The constructions above assume both translation-invariant weights and clean
open boundaries.  These assumptions can be relaxed in different ways without
identifying the two ends of the matrix.  We first let the band weights depend
on position, and then isolate fixed finite corner defects as a natural
boundary-state extension.

\subsection{Position-dependent band weights}

Neither Laplace construction fundamentally requires Toeplitz-constant
weights.  What Toeplitz constancy supplies is autonomy: it makes the local
coefficient minors and transfer weights independent of the position.  The
increasing-rows elimination survives for leading principal sections of an
arbitrary fixed infinite banded matrix, while the row-column construction
keeps the same signature graph and replaces one constant transfer matrix by a
position-dependent cocycle.

\begin{theorem}[Nonautonomous increasing-rows relation]\label{thm:variable-increasing}
Let $A_\infty=(\alpha_{ij})_{i,j\ge1}$ be a fixed infinite matrix satisfying
\[
 \alpha_{ij}=0\qquad\text{unless}\qquad -m_1\le j-i\le m_2,
\]
and let $A_n$ be its $n\times n$ leading principal section.  Put
\[
 d=\binom{m_1+m_2}{m_1}.
\]
For all sufficiently large $k$ there are coefficients
$c_0(k),\ldots,c_d(k)$, not all zero, such that
\begin{equation}\label{eq:variable-increasing}
 \sum_{\ell=0}^{d}c_\ell(k)\det A_{k+\ell}=0.
\end{equation}
The same statement holds with determinant replaced by permanent.  Thus both
sequences admit homogeneous nonautonomous annihilating relations involving at
most $d+1$ consecutive terms.
\end{theorem}

\begin{proof}
After transposition if necessary, assume $m_1\ge m_2$.  Repeat the proof of
Theorem~\ref{thm:increasing}, but retain the dependence of the coefficient
minors on the absolute position.  Expanding $A_{k+\ell}$ along its last
$m_1+\ell$ rows, bandedness leaves the same $d$ complementary cofactor
patterns $S\in\cS$.  Because all $A_{k+\ell}$ are leading principal sections
of one fixed infinite matrix, the complementary row set and complementary
column set associated with a given $S$ are independent of $\ell$.  Hence
\[
 \det A_{k+\ell}
   =\sum_{S\in\cS} B^{\det}_{\ell,S}(k)X_S(k),
 \qquad \ell=0,1,\ldots,d,
\]
where the $d$ cofactors $X_S(k)$ are common to all $d+1$ equations; only the
coefficient minors $B^{\det}_{\ell,S}(k)$ now depend on $k$.  The resulting
$(d+1)\times d$ coefficient matrix has a nonzero left-kernel vector
$(c_0(k),\ldots,c_d(k))$, which gives \eqref{eq:variable-increasing}.  The
signless Laplace expansion gives the permanent statement.
\end{proof}

Toeplitz translation invariance is therefore exactly what removes the $k$
dependence from the coefficient matrix in this construction and recovers the
constant-coefficient recurrence of Theorem~\ref{thm:increasing}.  The theorem
also explains why this extension is different from a moving corner-defect
problem: coherence under leading principal truncation preserves the same
complementary cofactors as $\ell$ varies.

For the row-column method, one obtains an equally concrete nonautonomous
transfer.  A convenient position-dependent diagonal model is
\[
 \widetilde A_k=(a_{j-i,\min(i,j)})_{i,j=1}^k,
 \qquad -m_1\le j-i\le m_2,
\]
with all other entries zero.  We write $a_s[t]$ for $a_{s,t}$.
Zakraj\v{s}ek and Petkov\v{s}ek \cite{Zakr} proved recurrence results for
principal minors of arbitrary banded matrices.  Here the point is more
concrete: the \emph{signature graph} of the row-column method is unchanged,
while its transition weights become functions of the position $k$.

For example, in the pentadiagonal case the same six states as in
Section~\ref{sec:rowcolumn} satisfy
\begin{equation}\label{eq:variableP}
 v_{k+1}=Q[k]v_k,
\end{equation}
where
\[
Q[k]=
\begin{bmatrix}
 a_0[k]&-a_{-1}[k-1]&a_{-2}[k-2]&0&0&0\\
 a_1[k]&0&0&-a_2[k-1]&0&0\\
 0&a_1[k]&0&0&-a_2[k-1]&0\\
 a_{-1}[k]&-a_{-2}[k-1]&0&0&0&0\\
 a_0[k+1]&0&0&0&0&-a_{-2}[k-1]\\
 a_2[k]&0&0&0&0&0
\end{bmatrix}.
\]
The permanent system uses the same matrix with the Laplace signs removed.
When the $a_s[t]$ are constant in $t$, \eqref{eq:variableP} reduces to
\eqref{eq:Pdet}.

Iterating \eqref{eq:variableP} gives
\[
 v_{k+j}=Q[k+j-1]\cdots Q[k]v_k.
\]
Taking the first coordinate for $j=0,1,\ldots,6$ produces seven linear forms
in the six state variables.  Eliminating the state vector therefore yields
a scalar variable-coefficient recurrence involving at most seven consecutive
determinants.  The same construction works for general $(m_1,m_2)$ once the
finite signature graph is known, and it applies equally to permanents.

This transfer formulation is preferable to printing the resulting scalar
coefficients: even in the pentadiagonal case those coefficients are lengthy,
whereas the $6\times6$ matrix $Q[k]$ is sparse and transparent.  It also
makes periodic and otherwise structured diagonal dependence accessible by
matrix products.  Finding effective reductions for structured position
dependence beyond the general finite-state construction remains open.

\subsection{Finite corner defects}\label{sec:finite-corner-defects}
Shevelev's survey explicitly records recurrence constructions for sequences
of determinants and permanents of matrices that differ from cyclic or
Toeplitz matrices in only finitely many positions \cite{SheSurvey1992}.
Thus recurrence existence for fixed finite defects is not an open issue here.
What remains interesting from the present constructive viewpoint is a sharp
finite-state description that keeps track of such defects without hiding them
inside a general recurrence-existence theorem.

After fixed row and column permutations, a finite defect scheme may be
collected into northeast and/or southwest corner blocks while leaving a clean
central Toeplitz band, as in the finite-defect literature
\cite{Koch,Kamenetskii2007,RAJK}.  Such preprocessing is not essential to the
row-column construction.  Direct Laplace expansion can instead pass through
the defects themselves.  While a northeast defect is active, expanding a
trailing defective column may delete additional leading rows; dually, a
southwest defect may force additional leading-column deletions.  This suggests
generalized two-sided signatures recording missing rows and columns at both
ends, together with a bounded relative shift of the surviving Toeplitz band
and finite data describing any still-active defect.

Because the defect blocks have fixed size, after a bounded initial phase no
new leading-boundary deletions are created; the trailing boundary then evolves
by the ordinary clean-band row-column rules.  The induced leading-boundary
data and band shift may persist, so one should not assume literal return to an
undeformed principal Toeplitz section, but all such data remain bounded.  This
is the natural finite-state architecture for a direct row-column extension.

The analogous extension of the increasing-rows method is less automatic.
Its clean-band elimination uses the same complementary cofactors for every
number of expanded rows; a corner defect can make those cofactors depend on
the expansion depth or introduce additional defect-containing cofactors that
must themselves be closed recursively.  A useful next problem is therefore
to develop a formal defect-state theory with sharp reachable/observable state
counts, sparse transition structure, minimal scalar recurrence orders, and a
precise comparison of the row-column and increasing-rows extensions.

\section{Cyclic closure and the full subset spectrum}\label{sec:cyclic}

The preceding extension changes the bulk weights.  Cyclic closure is
different: the bulk symbol is unchanged, but wrap-around couplings identify
the two boundaries.  The recurrence theory of circulant permanents has a long
history, beginning with early work of Minc \cite{Minc1964} and continuing
through his recurrence and permanental-compound constructions
\cite{Minc1985,Minc1987}, Shevelev's circulant recurrence program
\cite{SheFund1990,SheSurvey1992}, and the boundary-state transfer of Golin,
Leung, and Wang \cite{GolinLeungWang2006}.  On the determinant side,
Shevelev's canonical circulant representation is organized by associated
matrices of all exterior degrees \cite{SheSurvey1992,She1}.  Fourier
diagonalization gives a short route to the same all-subset spectrum.

Put $w=m_1+m_2$, let
\[
 q(z)=z^{m_1}a(z)=a_{m_2}\prod_{j=1}^w(z-z_j),
\]
and let $\Pi_n$ be the $n\times n$ cyclic shift matrix.  For $n>w$ set
\[
 C_n=a(\Pi_n)=\sum_{s=-m_1}^{m_2}a_s\Pi_n^s.
\]
The restriction $n>w$ merely avoids aliasing of the displayed diagonals.

\begin{proposition}[Cyclic determinant and generic minimality]\label{prop:cyclic}
For $n>w$,
\begin{equation}\label{eq:cyclic-det-product}
 \det C_n=(-1)^{m_2(n+1)}a_{m_2}^{\,n}
          \prod_{j=1}^{w}(1-z_j^n).
\end{equation}
Equivalently,
\begin{equation}\label{eq:cyclic-det-subsets}
 \det C_n=(-1)^{m_2}
 \sum_{J\subseteq[w]}(-1)^{|J|}
 \left(
  (-1)^{m_2}a_{m_2}\prod_{j\in J}z_j
 \right)^n.
\end{equation}
Hence the monic polynomial
\begin{equation}\label{eq:cyclic-annihilator}
 \chi_{\mathrm{cyc}}(t)=
 \prod_{J\subseteq[w]}
 \left(
  t-(-1)^{m_2}a_{m_2}\prod_{j\in J}z_j
 \right)
\end{equation}
annihilates the cyclic determinant sequence and has degree $2^w$.  If the
$2^w$ subset products $\prod_{j\in J}z_j$ are pairwise distinct, then this
degree is minimal.
\end{proposition}

\begin{proof}
The eigenvalues of $\Pi_n$ are the $n$th roots of unity.  Thus
\[
 \det C_n=\prod_{\zeta^n=1}a(\zeta)
 =a_{m_2}^{\,n}
   \prod_{\zeta^n=1}\zeta^{-m_1}
   \prod_{j=1}^w\prod_{\zeta^n=1}(\zeta-z_j).
\]
Since
\[
 \prod_{\zeta^n=1}\zeta=(-1)^{n+1},
 \qquad
 \prod_{\zeta^n=1}(\zeta-z)=(-1)^{n+1}(1-z^n),
\]
and $w=m_1+m_2$, the total sign is $(-1)^{m_2(n+1)}$, which proves
\eqref{eq:cyclic-det-product}.  Expanding the product gives
\eqref{eq:cyclic-det-subsets}, hence \eqref{eq:cyclic-annihilator} is an
annihilator.  When the subset products are pairwise distinct, the right side
of \eqref{eq:cyclic-det-subsets} is a linear combination of $2^w$ distinct
exponential sequences with nonzero coefficients.  Their Vandermonde matrix
is nonsingular, so no annihilating polynomial of smaller degree exists.
\end{proof}

\begin{remark}[Comparison with Widom's formula]\label{rem:cyclic-exterior}
The characteristic bases in Proposition~\ref{prop:cyclic} are the same root
products
\[
 W_J=(-1)^{m_2}a_{m_2}\prod_{j\in J}z_j
\]
that occur in Widom's formula \eqref{eq:widomform}.  What changes is the
admissible family of subsets.  For open Toeplitz sections only the layer
$|J|=m_2$ occurs, giving generically
$\binom{w}{m_2}$ modes; cyclic closure uses every $J\subseteq[w]$ and hence
all $2^w$ modes.  In exterior-algebra language this is the passage
\[
 \bigwedge^{m_2}V
 \qquad\longrightarrow\qquad
 \bigoplus_{r=0}^w\bigwedge^rV
 =\bigwedge\nolimits^\bullet V,
 \qquad
 \sum_{r=0}^w\binom wr=2^w.
\]
This is the all-associated-degree organization appearing in Shevelev's
circulant representation \cite{SheSurvey1992}; Proposition~\ref{prop:cyclic}
adds that the full subset count is generically minimal for determinants.
\end{remark}

For permanents there is no analogous Fourier product factorization in
general.  Standard cut-open cycle-cover constructions nevertheless give a
finite-state transfer for each fixed jump set; see
\cite{GolinLeungWang2006}.  Their boundary states retain the partial
in-degree and out-degree data that survive across the cut, and the resulting
transfer can be much larger than the minimal scalar recurrence.  Indeed,
substantial characteristic-polynomial reductions are a central feature of
that construction.  Thus the determinant's exact all-subset count $2^w$
should not be read as a corresponding generic minimality statement for
cyclic permanents.  Minc's recurrence constructions
\cite{Minc1985,Minc1987} and their transfer interpretation in
\cite{GolinLeungWang2006} provide important examples of such reductions.
Sparse circulants also admit complementary determinantal reductions in
special cases \cite{CodenottiResta2002}.

The open and cyclic determinant families thus provide a particularly clean
comparison.  The same symbol-root products govern both, but open boundary
conditions select one exterior degree while cyclic closure collects all
exterior degrees.  The two generic mode counts are respectively
$\binom{m_1+m_2}{m_1}$ and $2^{m_1+m_2}$.

\section{Concluding remarks}\label{sec:conclusion}

The two Laplace constructions give complementary realizations of the same
binomial-scale recurrence mechanism.  The increasing-rows method gives a
direct cofactor-elimination argument, while the row-column method produces a
sparse transfer with exactly
\[
 \binom{m_1}{j}\binom{m_2}{j}
\]
states at level $j$ and total dimension
$\binom{m_1+m_2}{m_1}$.  This exact state classification also gives the
transition sparsity formula and explains the central binomial scale in the
balanced case.

For determinants, the two constructions connect to the classical compound
companion representation.  The increasing-rows cofactors are compound
coordinates, and the row-column transfer has the Widom characteristic
polynomial and is generically similar to the same compound transfer.  Hence
the binomial recurrence order is generically minimal for unrestricted
fixed-band Toeplitz determinants.  For permanents, the same open-boundary
state graph gives the binomial upper bound, but minimality remains a separate
problem and may depend strongly on the support and weights.

The extensions considered here preserve this distinction.  Position-dependent
band weights retain the finite signature graph but replace the autonomous
transfer by a cocycle.  Cyclic closure changes the boundary condition and,
for determinants, replaces one fixed exterior degree by the full exterior
algebra, increasing the generic mode count to $2^{m_1+m_2}$.  The sparse
Toeplitz--Hessenberg example illustrates a complementary use of the scalar
recurrence: after cubic decomposition it organizes the finite-section
spectrum through a threefold-symmetric $2$-orthogonal polynomial family.

Natural next problems are to characterize systematic reductions of permanent
transfers, to scalarize structured nonautonomous cocycles efficiently, and to
construct minimal scalar recurrences without first forming a full
characteristic polynomial.  Fixed finite corner defects provide a concrete
test case: recurrence existence is known, but sharp augmented state spaces
and observable quotients remain to be determined.

\section*{Code availability}
Implementations of the recurrence constructions in Mathematica and SageMath
are available at \url{https://github.com/maxale/matrix_codes}.

\section*{Declaration of generative AI and AI-assisted technologies in the manuscript preparation process}
During the preparation of this work, the authors used ChatGPT (OpenAI) to
assist with literature searches, manuscript organization, and language and
editorial refinement.  After using this tool, the authors reviewed and edited
the content as needed and take full responsibility for the content of the
publication.

\printbibliography

@article{Sweet,
  author = {Sweet, R. A.},
  title = {A recursive relation for the determinant of a pentadiagonal matrix},
  journal = {Communications of the ACM},
  volume = {12},
  number = {6},
  pages = {330--332},
  year = {1969},
  doi = {10.1145/363011.363152}
}

@article{Cinkir,
  author = {Cinkir, Z.},
  title = {An elementary algorithm for computing the determinant of pentadiagonal {Toeplitz} matrices},
  journal = {Journal of Computational and Applied Mathematics},
  volume = {236},
  number = {9},
  pages = {2298--2305},
  year = {2012},
  doi = {10.1016/j.cam.2011.11.017}
}

@article{Jia,
  author = {Jia, J. and Yang, B. and Li, S.},
  title = {On a homogeneous recurrence relation for the determinants of general pentadiagonal {Toeplitz} matrices},
  journal = {Computers \& Mathematics with Applications},
  volume = {71},
  number = {4},
  pages = {1036--1044},
  year = {2016},
  doi = {10.1016/j.camwa.2016.01.027}
}

@article{An,
  author = {An{\dj}eli\'c, M. and da Fonseca, C. M.},
  title = {Some determinantal considerations for pentadiagonal matrices},
  journal = {Linear and Multilinear Algebra},
  volume = {69},
  number = {16},
  pages = {3121--3129},
  year = {2021},
  doi = {10.1080/03081087.2019.1708845}
}

@article{DuFonsecaPereira2022,
  author = {Du, Z. and da Fonseca, C. M. and Pereira, A.},
  title = {On determinantal recurrence relations of banded matrices},
  journal = {Kuwait Journal of Science},
  volume = {49},
  number = {1},
  pages = {1--9},
  year = {2022},
  doi = {10.48129/kjs.v49i1.11165}
}

@article{Widom1958,
  author = {Widom, H.},
  title = {On the eigenvalues of certain {Hermitian} operators},
  journal = {Transactions of the American Mathematical Society},
  volume = {88},
  number = {2},
  pages = {491--522},
  year = {1958},
  doi = {10.1090/S0002-9947-1958-0098321-8}
}

@article{SheFund1990,
  author = {Shevelev, V. S.},
  title = {The fundamental theorem on sequences of permanents of circulants generated by $k$-dimensional arithmetical vectors and its applications},
  journal = {Doklady Akademii Nauk Ukrainskoi SSR, Series A},
  number = {9},
  year = {1990}
}

@article{She,
  author = {Shevelev, V. S.},
  title = {Recurrence formulas for permanents and determinants of {Toeplitz} matrices},
  journal = {Doklady Akademii Nauk Ukrainskoi SSR},
  number = {10},
  pages = {32--36},
  year = {1991}
}

@article{She1,
  author = {Shevelev, V. S.},
  title = {Modern enumeration theory of permutations with restricted positions},
  journal = {Discrete Mathematics and Applications},
  volume = {3},
  number = {3},
  pages = {229--264},
  year = {1993},
  doi = {10.1515/dma.1993.3.3.229}
}

@article{SheSurvey1992,
  author = {Shevelev, V. S.},
  title = {Some problems of the theory of enumerating the permutations with restricted positions},
  journal = {Journal of Soviet Mathematics},
  volume = {61},
  number = {4},
  pages = {2272--2317},
  year = {1992},
  doi = {10.1007/BF01104103}
}

@article{Zakr,
  author = {Zakraj\v{s}ek, H. and Petkov\v{s}ek, M.},
  title = {{Pascal}-like determinants are recursive},
  journal = {Advances in Applied Mathematics},
  volume = {33},
  number = {3},
  pages = {431--450},
  year = {2004},
  doi = {10.1016/j.aam.2003.09.004}
}

@article{Alexandersson2012,
  author = {Alexandersson, P.},
  title = {{Schur} polynomials, banded {Toeplitz} matrices and {Widom}'s formula},
  journal = {The Electronic Journal of Combinatorics},
  volume = {19},
  number = {4},
  pages = {P22},
  year = {2012},
  doi = {10.37236/2651}
}

@article{Getu1991,
  author = {Getu, S.},
  title = {Evaluating determinants via generating functions},
  journal = {Mathematics Magazine},
  volume = {64},
  number = {1},
  pages = {45--53},
  year = {1991},
  doi = {10.1080/0025570X.1991.11977573}
}

@article{Cod,
  author = {Codenotti, B. and Crespi, V. and Resta, G.},
  title = {On the permanent of certain (0,1) {Toeplitz} matrices},
  journal = {Linear Algebra and its Applications},
  volume = {267},
  pages = {65--100},
  year = {1997},
  doi = {10.1016/S0024-3795(97)80043-2}
}

@article{Schwartz2009,
  author = {Schwartz, M.},
  title = {Efficiently computing the permanent and hafnian of some banded {Toeplitz} matrices},
  journal = {Linear Algebra and its Applications},
  volume = {430},
  number = {4},
  pages = {1364--1374},
  year = {2009},
  doi = {10.1016/j.laa.2008.10.029}
}

@article{Koch,
  author = {Kocharovsky, V. V. and Kocharovsky, V. V. and Martyanov, V. V. and Tarasov, S. V.},
  title = {Exact recursive calculation of circulant permanents: A band of different diagonals inside a uniform matrix},
  journal = {Entropy},
  volume = {23},
  number = {11},
  pages = {1423},
  year = {2021},
  doi = {10.3390/e23111423}
}

@article{Kamenetskii2007,
  author = {Kamenetskii, A. M.},
  title = {Rationality of generating functions of cycle rook polynomials and cycle permanents of {Toeplitz} matrices and circulants},
  journal = {Russian Mathematical Surveys},
  volume = {62},
  number = {6},
  pages = {1207--1209},
  year = {2007},
  doi = {10.1070/RM2007v062n06ABEH004487}
}

@inproceedings{GolinLeungWang2006,
  author = {Golin, Mordecai J. and Leung, Yiu Cho and Wang, Yajun},
  title = {Permanents of Circulants: A Transfer Matrix Approach},
  booktitle = {Proceedings of the 8th Workshop on Algorithm Engineering and Experiments and the 3rd Workshop on Analytic Algorithms and Combinatorics},
  pages = {263--272},
  publisher = {Society for Industrial and Applied Mathematics},
  year = {2006},
  doi = {10.1137/1.9781611972962.11},
  eprint = {0708.0907},
  archivePrefix = {arXiv},
  primaryClass = {math.CO}
}

@misc{Hendel2026,
  author = {Hendel, R. J.},
  title = {Proof of convergence of a {Laplace} expansion algorithm for calculating recursions satisfied by a family of determinants},
  year = {2026},
  eprint = {2601.04454},
  archivePrefix = {arXiv},
  primaryClass = {math.CO},
  doi = {10.48550/arXiv.2601.04454},
  note = {arXiv preprint, version 2}
}

@article{NoschesePasquiniReichel2013,
  author = {Noschese, S. and Pasquini, L. and Reichel, L.},
  title = {Tridiagonal {Toeplitz} matrices: Properties and novel applications},
  journal = {Numerical Linear Algebra with Applications},
  volume = {20},
  number = {2},
  pages = {302--326},
  year = {2013},
  doi = {10.1002/nla.1811}
}

@book{ZimmermannEtAl2018,
  author = {Zimmermann, P. and Casamayou, A. and Cohen, N. and Connan, G. and Dumont, T. and others},
  title = {Computational Mathematics with {SageMath}},
  publisher = {SIAM},
  address = {Philadelphia},
  year = {2018},
  isbn = {978-1-61197-545-1},
  doi = {10.1137/1.9781611975468}
}

@article{Fid,
  author = {Fiduccia, C. M.},
  title = {An efficient formula for linear recurrences},
  journal = {SIAM Journal on Computing},
  volume = {14},
  number = {1},
  pages = {106--112},
  year = {1985},
  doi = {10.1137/0214007}
}

@article{Kho,
  author = {Khomovsky, D. I.},
  title = {Efficient computation of terms of linear recurrence sequences of any order},
  journal = {Integers},
  volume = {18},
  pages = {A39},
  year = {2018}
}

@inproceedings{BostanMori2021,
  author = {Bostan, A. and Mori, R.},
  title = {A simple and fast algorithm for computing the n-th term of a linearly recurrent sequence},
  booktitle = {Proceedings of the 2021 Symposium on Simplicity in Algorithms (SOSA)},
  pages = {118--132},
  year = {2021},
  doi = {10.1137/1.9781611976496.14}
}

@article{BiniCapovani1983,
  author = {Bini, D. and Capovani, M.},
  title = {Spectral and computational properties of band symmetric {Toeplitz} matrices},
  journal = {Linear Algebra and its Applications},
  volume = {52--53},
  pages = {99--126},
  year = {1983},
  doi = {10.1016/0024-3795(83)80009-3}
}

@article{Trench1985,
  author = {Trench, W. F.},
  title = {On the eigenvalue problem for {Toeplitz} band matrices},
  journal = {Linear Algebra and its Applications},
  volume = {64},
  pages = {199--214},
  year = {1985},
  doi = {10.1016/0024-3795(85)90277-0}
}

@article{Tismenetsky1987,
  author = {Tismenetsky, M.},
  title = {Determinant of block-{Toeplitz} band matrices},
  journal = {Linear Algebra and its Applications},
  volume = {85},
  pages = {165--184},
  year = {1987},
  doi = {10.1016/0024-3795(87)90214-X}
}

@article{DouakMaroni1997I,
  author = {Douak, K. and Maroni, P.},
  title = {On $d$-orthogonal {Tchebychev} polynomials, {I}},
  journal = {Applied Numerical Mathematics},
  volume = {24},
  number = {1},
  pages = {23--53},
  year = {1997},
  doi = {10.1016/S0168-9274(97)00006-8}
}

@article{BenRomdhane2008,
  author = {Ben Romdhane, N.},
  title = {On the zeros of $d$-symmetric $d$-orthogonal polynomials},
  journal = {Journal of Mathematical Analysis and Applications},
  volume = {344},
  number = {2},
  pages = {888--897},
  year = {2008},
  doi = {10.1016/j.jmaa.2008.02.038}
}

@article{LoureiroVanAssche2020,
  author = {Loureiro, A. F. and Van Assche, W.},
  title = {Threefold symmetric {Hahn}-classical multiple orthogonal polynomials},
  journal = {Analysis and Applications},
  volume = {18},
  number = {2},
  pages = {271--332},
  year = {2020},
  doi = {10.1142/S0219530519500106}
}

@article{Minc1985,
  author = {Minc, Henryk},
  title = {Recurrence formulas for permanents of $(0,1)$-circulants},
  journal = {Linear Algebra and its Applications},
  volume = {71},
  pages = {241--265},
  year = {1985},
  doi = {10.1016/0024-3795(85)90250-2}
}

@article{Minc1987,
  author = {Minc, Henryk},
  title = {Permanental compounds and permanents of $(0,1)$-circulants},
  journal = {Linear Algebra and its Applications},
  volume = {86},
  pages = {11--42},
  year = {1987},
  doi = {10.1016/0024-3795(87)90285-0}
}

@book{BottcherGrudsky2005,
  author = {B\"ottcher, Albrecht and Grudsky, Sergei M.},
  title = {Spectral Properties of Banded {Toeplitz} Matrices},
  publisher = {Society for Industrial and Applied Mathematics},
  address = {Philadelphia},
  year = {2005},
  isbn = {978-0-89871-599-6},
  doi = {10.1137/1.9780898717853}
}

@article{Minc1964,
  author = {Minc, Henryk},
  title = {Permanents of $(0,1)$-circulants},
  journal = {Canadian Mathematical Bulletin},
  volume = {7},
  number = {2},
  pages = {253--263},
  year = {1964},
  doi = {10.4153/CMB-1964-023-3}
}

@article{CodenottiResta2002,
  author = {Codenotti, B. and Resta, G.},
  title = {Computation of sparse circulant permanents via determinants},
  journal = {Linear Algebra and its Applications},
  volume = {355},
  pages = {15--34},
  year = {2002},
  doi = {10.1016/S0024-3795(02)00330-0}
}

@article{RAJK,
  author = {Rajkovi\'c, P. and Radovi\'c, L. and Rajkovi\'c, J.},
  title = {Almost multi-diagonal determinants},
  journal = {Kragujevac Journal of Mathematics},
  volume = {47},
  number = {7},
  pages = {1047--1056},
  year = {2023},
  doi = {10.46793/KgJMat2307.1047R}
}

@article{ShapiroStampach2019,
  author = {Shapiro, Boris and {\v S}tampach, Franti{\v s}ek},
  title = {Non-self-adjoint {Toeplitz} matrices whose principal submatrices have real spectrum},
  journal = {Constructive Approximation},
  volume = {49},
  pages = {191--226},
  year = {2019},
  doi = {10.1007/s00365-017-9408-0}
}

@article{MarrVineyard1988,
  author = {Marr, Robert B. and Vineyard, George H.},
  title = {Five-diagonal {Toeplitz} determinants and their relation to {Chebyshev} polynomials},
  journal = {SIAM Journal on Matrix Analysis and Applications},
  volume = {9},
  number = {4},
  pages = {579--586},
  year = {1988},
  doi = {10.1137/0609048}
}

@article{HadjElouafi2008,
  author = {Hadj, A. D. A. and Elouafi, M.},
  title = {On the characteristic polynomial, eigenvectors and determinant of a pentadiagonal matrix},
  journal = {Applied Mathematics and Computation},
  volume = {198},
  number = {2},
  pages = {634--642},
  year = {2008},
  doi = {10.1016/j.amc.2007.09.005}
}

@article{McMillen2009,
  author = {McMillen, Tyler},
  title = {On the eigenvalues of double band matrices},
  journal = {Linear Algebra and its Applications},
  volume = {431},
  number = {10},
  pages = {1890--1897},
  year = {2009},
  doi = {10.1016/j.laa.2009.06.026}
}

@article{EgervarySzasz1928,
  author = {Egerv{\'a}ry, E. and Sz{\'a}sz, O.},
  title = {Einige Extremalprobleme im Bereiche der trigonometrischen Polynome},
  journal = {Mathematische Zeitschrift},
  volume = {27},
  pages = {641--652},
  year = {1928},
  doi = {10.1007/BF01171120}
}

@article{Losonczi1992,
  author = {Losonczi, L.},
  title = {Eigenvalues and eigenvectors of some tridiagonal matrices},
  journal = {Acta Mathematica Hungarica},
  volume = {60},
  number = {3--4},
  pages = {309--322},
  year = {1992},
  doi = {10.1007/BF00051649}
}

@article{DaFonsecaYilmaz2015,
  author = {da Fonseca, Carlos M. and Y{\i}lmaz, Fatih},
  title = {Some comments on $k$-tridiagonal matrices: determinant, spectra, and inversion},
  journal = {Applied Mathematics and Computation},
  volume = {270},
  pages = {644--647},
  year = {2015},
  doi = {10.1016/j.amc.2015.08.088}
}

\clearpage
\appendix
\section{Balanced \texorpdfstring{$(3,3)$}{(3,3)} row-column transfer matrix}\label{app:33}

This appendix records one larger instance of the row-column construction.
It makes the state ordering behind the sparse transfer matrix explicit and
provides a concrete $20$-state example of the general state classification.
Theorem~\ref{thm:states} explains a priori why there are $\binom63=20$
states and $50$ nonzero transition entries.

Let $\sigma_1,\ldots,\sigma_{20}$ be the signatures, in the order
\begin{align*}
&\varnothing,
\begin{bmatrix}1\\2\end{bmatrix},
\begin{bmatrix}1\\3\end{bmatrix},
\begin{bmatrix}1\\4\end{bmatrix},
\begin{bmatrix}2\\1\end{bmatrix},
\begin{bmatrix}3\\1\end{bmatrix},
\begin{bmatrix}2\\2\end{bmatrix},
\begin{bmatrix}3\\2\end{bmatrix},
\begin{bmatrix}2\\3\end{bmatrix},
\begin{bmatrix}3\\3\end{bmatrix},\\
&\begin{bmatrix}1,2\\2,3\end{bmatrix},
\begin{bmatrix}1,2\\2,4\end{bmatrix},
\begin{bmatrix}1,3\\2,3\end{bmatrix},
\begin{bmatrix}1,3\\2,4\end{bmatrix},
\begin{bmatrix}2,3\\1,3\end{bmatrix},
\begin{bmatrix}2,4\\1,3\end{bmatrix},
\begin{bmatrix}1,3\\3,4\end{bmatrix},
\begin{bmatrix}2,3\\1,2\end{bmatrix},
\begin{bmatrix}2,3\\2,3\end{bmatrix},
\begin{bmatrix}1,2,3\\2,3,4\end{bmatrix}.
\end{align*}
Thus, with
\[
 v_k=\bigl(M_{\sigma_1}(k),\ldots,M_{\sigma_{20}}(k)\bigr)^{\mathsf T},
\]
Algorithm~\ref{alg:rowcolumn} gives
\[
 v_{k+1}=Q_{3,3}v_k,
\]
where the following matrix has exactly $50$ nonzero entries.

\begingroup
\setcounter{MaxMatrixCols}{20}
\setlength{\arraycolsep}{2.2pt}
\renewcommand{\arraystretch}{1.05}
\[
\resizebox{\textwidth}{!}{$
Q_{3,3}=\begin{bmatrix}
a_0 & -a_{-1} & a_{-2} & -a_{-3} & 0 & 0 & 0 & 0 & 0 & 0 & 0 & 0 & 0 & 0 & 0 & 0 & 0 & 0 & 0 & 0 \\
a_1 & 0 & 0 & 0 & -a_2 & a_3 & 0 & 0 & 0 & 0 & 0 & 0 & 0 & 0 & 0 & 0 & 0 & 0 & 0 & 0 \\
0 & a_1 & 0 & 0 & 0 & 0 & -a_2 & a_3 & 0 & 0 & 0 & 0 & 0 & 0 & 0 & 0 & 0 & 0 & 0 & 0 \\
0 & 0 & a_1 & 0 & 0 & 0 & 0 & 0 & -a_2 & a_3 & 0 & 0 & 0 & 0 & 0 & 0 & 0 & 0 & 0 & 0 \\
a_{-1} & -a_{-2} & a_{-3} & 0 & 0 & 0 & 0 & 0 & 0 & 0 & 0 & 0 & 0 & 0 & 0 & 0 & 0 & 0 & 0 & 0 \\
0 & 0 & 0 & 0 & a_{-1} & 0 & -a_{-2} & 0 & a_{-3} & 0 & 0 & 0 & 0 & 0 & 0 & 0 & 0 & 0 & 0 & 0 \\
a_0 & 0 & 0 & 0 & 0 & 0 & 0 & 0 & 0 & 0 & -a_{-2} & a_{-3} & 0 & 0 & 0 & 0 & 0 & 0 & 0 & 0 \\
0 & 0 & 0 & 0 & a_0 & 0 & 0 & 0 & 0 & 0 & 0 & 0 & -a_{-2} & a_{-3} & 0 & 0 & 0 & 0 & 0 & 0 \\
0 & a_0 & 0 & 0 & 0 & 0 & 0 & 0 & 0 & 0 & 0 & 0 & 0 & 0 & -a_2 & a_3 & 0 & 0 & 0 & 0 \\
0 & 0 & 0 & 0 & 0 & 0 & a_0 & 0 & 0 & 0 & 0 & 0 & -a_{-1} & 0 & 0 & 0 & a_{-3} & 0 & 0 & 0 \\
a_2 & 0 & 0 & 0 & -a_3 & 0 & 0 & 0 & 0 & 0 & 0 & 0 & 0 & 0 & 0 & 0 & 0 & 0 & 0 & 0 \\
0 & a_2 & 0 & 0 & 0 & 0 & -a_3 & 0 & 0 & 0 & 0 & 0 & 0 & 0 & 0 & 0 & 0 & 0 & 0 & 0 \\
a_1 & 0 & 0 & 0 & 0 & 0 & 0 & 0 & 0 & 0 & 0 & 0 & 0 & 0 & 0 & 0 & 0 & -a_3 & 0 & 0 \\
0 & a_1 & 0 & 0 & 0 & 0 & 0 & 0 & 0 & 0 & 0 & 0 & 0 & 0 & -a_3 & 0 & 0 & 0 & 0 & 0 \\
a_{-1} & 0 & 0 & 0 & 0 & 0 & 0 & 0 & 0 & 0 & -a_{-3} & 0 & 0 & 0 & 0 & 0 & 0 & 0 & 0 & 0 \\
0 & 0 & 0 & 0 & a_{-1} & 0 & 0 & 0 & 0 & 0 & 0 & 0 & -a_{-3} & 0 & 0 & 0 & 0 & 0 & 0 & 0 \\
0 & 0 & 0 & 0 & 0 & 0 & 0 & 0 & 0 & 0 & a_1 & 0 & 0 & 0 & 0 & 0 & 0 & 0 & -a_3 & 0 \\
a_{-2} & -a_{-3} & 0 & 0 & 0 & 0 & 0 & 0 & 0 & 0 & 0 & 0 & 0 & 0 & 0 & 0 & 0 & 0 & 0 & 0 \\
a_0 & 0 & 0 & 0 & 0 & 0 & 0 & 0 & 0 & 0 & 0 & 0 & 0 & 0 & 0 & 0 & 0 & 0 & 0 & -a_{-3} \\
a_3 & 0 & 0 & 0 & 0 & 0 & 0 & 0 & 0 & 0 & 0 & 0 & 0 & 0 & 0 & 0 & 0 & 0 & 0 & 0
\end{bmatrix}.
$}
\]
\endgroup

Let
\[
 \det(tI-Q_{3,3})=t^{20}-r_1t^{19}-r_2t^{18}-\cdots-r_{20}.
\]
Then Cayley--Hamilton gives the order-$20$ recurrence
\begin{equation}\label{eq:app33-rec}
 D_k(3,3)=\sum_{i=1}^{20}r_iD_{k-i}(3,3).
\end{equation}

\end{document}